\documentclass[1p]{elsarticle}
\usepackage{hyperref}
\usepackage{amssymb}
\usepackage{latexsym}
\usepackage{amscd}
\usepackage{amsthm}
\usepackage{amsfonts}
\usepackage{CJK}
\usepackage{CJKulem}
\usepackage{color}
\usepackage{colortbl}
\usepackage{fancyhdr}
\usepackage{indentfirst}
\usepackage{lastpage}
\usepackage{latexsym}
\usepackage{listings}
\usepackage{multirow}
\usepackage{mathrsfs}
\usepackage{placeins}
\usepackage{titlesec}
\usepackage{ulem}
\usepackage{graphicx} %We can use any other package if necessary
\usepackage{amsmath} % If you want to use AMS-LaTeX, comment out these
\usepackage{amssymb} % two commands.
\usepackage{yfonts}
\usepackage{xypic}
\usepackage{bm}

\newdimen\AAdi%
\newbox\AAbo%
\def\AAk#1#2{\s_etbox\AAbo=\hbox{#2}\AAdi=\wd\AAbo\kern#1\AAdi{}}%
\def\AAr#1#2#3{\s_etbox\AAbo=\hbox{#2}\AAdi=\ht\AAbo\raise#1\AAdi\hbox{#3}}%
\font\tenmsb=msbm10 at 12pt \font\sevenmsb=msbm7 at 8pt
\font\fivemsb=msbm5 at 6pt
\newfam\msbfam
\textfont\msbfam=\tenmsb \scriptfont\msbfam=\sevenmsb
\scriptscriptfont\msbfam=\fivemsb
\def\Bbb#1{{\tenmsb\fam\msbfam#1}}
\newtheorem{thm}{Theorem}[section]
\newtheorem{main-thm}{Main Theorem}
\newtheorem{lem}[thm]{Lemma}
\newtheorem{cor}[thm]{Corollary}
\newtheorem{pro}[thm]{Proposition}

\newtheorem{con}[thm]{Conjecture}
\theoremstyle{remark}\newtheorem{rem}{Remark}[section]

\newcommand{\Section}[2]{\setcounter{equation}{0}
	\allowdisplaybreaks
	\section[#1]{#2}}

\def\n{\nabla}

\def\f#1#2{\frac{#1}{#2}}

\def\grs#1#2{\bold G_{#1,#2}}

\def\mc#1{\mathcal{#1}}

\def\pf#1{\frac{\partial}{\partial #1}}

\def\a{\alpha}
\def\be{\beta}

\def\p#1{\partial #1}

\def\de{\delta}
\def\De{\Delta}

\def\G{\Gamma}
\def\g{\gamma}

\def\la{\lambda}

\def\th{\theta}
\def\Th{\Theta}
\def\si{\sigma}
\def\Si{\Sigma}

\def\w{\wedge}

\def\R{\Bbb{R}}
\def\C{\Bbb{C}}

\def\lan{\langle}
\def\ran{\rangle}
\def\ra{\rightarrow}

\def\ol{\overline}

\begin{document}
	\title
	{ The rigidity of minimal Legendrian submanifolds in the Euclidean spheres via eigenvalues of fundamental matrices}
	
	\author[f]{Pei-Yi Wu}
	\ead{17110180010@fudan.edu.cn}
	\author[f,m]{Ling Yang}
	\ead{yanglingfd@fudan.edu.cn}
	
	\address[f]{School of Mathematical Sciences, Fudan University, Shanghai, 200433, China}
	\address[m]{Shanghai Center for Mathematical Sciences, Shanghai, 200438, China}

	%\thanks{This work was supported in part by NSFC (Grant No. 11622103)}

	\begin{abstract}
    In this paper, we study the rigidity problem for compact minimal Legendrian submanifolds in the unit Euclidean spheres
    via eigenvalues of fundamental matrices, which measure the squared norms of the second fundamental form on
    all normal directions. By using Lu's inequality \cite{Lu} on the upper bound of the squared norm
    of Lie brackets of symmetric matrices, we establish an optimal pinching theorem for such submanifolds of all dimensions,
    giving a new characterization for the Calabi tori. This pinching condition can also be described by the eigenvalues of the Ricci curvature tensor.
    Moreover, when the third large eigenvalue of the fundamental matrix vanishes everywhere, we
    get an optimal rigidity theorem under a weaker pinching condition.

		%\begin{keyword}

		%\end{keyword}
		%\begin{keyword}
		%	\sep
		%	\sep
		%	\sep
		%	\sep
		%	\\{\ }\\
		%	\MSC[2010] 53A10, 53C42, 53C45, 32H25, 30C15, 32A22, 51B20.
		%\end{keyword}
	\end{abstract}

	%\begin{keyword}
	%\sep  \sep \sep
	%\end{keyworkd}
	
	\maketitle
	
	\tableofcontents
	
	\renewcommand{\proofname}{\it Proof.}
	
	%\small \parskip0.1mm\tableofcontents \normalsize\parskip3mm
	
	\Section{Introduction}{Introduction}

    In 1968, J. Simons \cite{S} proved a well-known rigidity theorem as follows:
    \begin{thm}
    Let $M$ be an $n$-dimensional compact minimal submanifold
    in $S^{n+m}$, denote by $|B|^2$ be the square norm of the second fundamental form $B$ of $M$, then
    \begin{equation}
    \int_M |B|^2\left[(2-\f{1}{m})|B|^2-n\right]dM\geq 0.
    \end{equation}
    As a corollary, the pinching condition $0\leq |B|^2\leq \f{n}{2-\f{1}{m}}$ forces $|B|^2\equiv 0$ or $|B|^2\equiv \f{n}{2-\f{1}{m}}$.
    \end{thm}
    As shown by Chern-do Carmo-Kobayashi \cite{C-D-K} and B. Lawson \cite{L}, $|B|^2\equiv \f{n}{2-\f{1}{m}}$ means $M$ is a Clifford torus or a Veronese surface in $S^4$.
    By the Gauss equation, the scalar curvature $R$ of $M$ is completely determined by $|B|^2$ (i.e. $R=n(n-1)-|B|^2$), so Simons' theorem can be seen as an intrinsic
    rigidity result on the scalar curvature. Based on this phenomenon, S. S. Chern \cite{C} raised a well-known conjecture as follows,
    which has been listed by S. T. Yau \cite{Y} as one of the 120 open problems in the field of differential geometry.
    \begin{con} Let $M^n$
    be a compact minimal submanifold in $S^{n+m}$ with constant squared norm of the second fundamental form, then the value of $|B|^2$ must lies in a discrete subset
    of $\R$.
    \end{con}
    From this viewpoint, Simons' theorem describes the first gap of $|B|^2$. For the hypersurface cases (i.e. $m=1$), Peng-Terng \cite{P-T,P-T2} made the first
    effort to the Chern conjecture and confirmed the second gap of $|B|^2$. This beautiful work attracts a lot of successive studies,
    see \cite{Ch,Y-C1,Y-C2,Y-C3,S-Y,W-X,Z,D-X,X-X,L-X-X}. On the other hand, for the higher codimensional cases,
    Li-Li \cite{L-L} and Chen-Xu \cite{C-X} independently got a rigidity theorem whose condition is weaker than Simons' theorem:
   \begin{thm}\label{Li}
   Let $M$ be an $n$-dimensional compact minimal submanifold
    in $S^{n+m}$ with $m\geq 2$. If $|B|^2\leq \f{2n}{3}$, then $M$ is a totally geodesic subsphere ($|B|^2\equiv 0$),
     or the Veronese surface (here $n=m=2$ and $|B|^2\equiv \f{4}{3}$).
   \end{thm}
   For each $p\in M$, let $\{e_1,\cdots,e_n\}$ and $\{\nu_1,\cdots,\nu_m\}$ be orthonormal basis of the tangent space
   and the normal space at $p$, respectively, then
   \begin{equation}
   (S_{\a\be}):=\left(\sum_{i,j}\lan B_{e_ie_j},\nu_\a\ran \lan B_{e_ie_j},\nu_\be\ran\right)
   \end{equation}
   is called the {\it fundamental matrix} at $p$. Z. Q. Lu \cite{Lu} studied the rigidity problem via eigenvalues of fundamental matrices and established
   the following pinching theorem:
\begin{thm}\label{Lu4}
Let $M$ be an $n$-dimensional compact minimal submanifold in $S^{n+m}$ and $\la_2$ be the second large eigenvalue of the fundamental matrix
at each point. If $|B|^2+\la_2\leq n$, then $M$ is a totally geodesic subsphere ($|B|^2+\la_2\equiv 0$),
a Clifford torus ($|B|^2\equiv n$ and $\la_2\equiv 0$) or the Veronese surface ($|B|^2+\la_2\equiv 2$).
\end{thm}
Observing that $|B|^2=\text{tr}(S_{\a\be})=\sum\limits_{\a} \la_\a$, $|B|^2\leq \f{2n}{3}$ implies $|B|^2+\la_2\leq n$ and hence
Theorem \ref{Lu4} is an improvement of Theorem \ref{Li}. However, up to now, for the cases of $n\geq 3$ or $m\geq 3$, it is unknown whether there exists
a pinching condition forcing $M$ to be a non-totally-geodesic minimal submanifold.

Given a submanifold $M^n\subset S^{n+m}$, the cone $CM$ over $M$ is defined as
\begin{equation}
CM:=\{tx\in \R^{n+m+1}:t\in \R,x\in M\},
\end{equation}
which turns to be a minimal submanifold of $\R^{n+m+1}$ whenever $M$ is minimal (see e.g. \S 1.4 of \cite{X}). When $m=n+1$,
$M$ is called a {\it Legendrian submanifold} if and only if $CM$ is a {\it Lagrangian submanifold} of
$\R^{2n+2}=\C^{n+1}$, i.e. the complex structure $J$ of $\C^{n+1}$ carries each tangent space of $M$ onto its corresponding normal space.
On the other hand, let
\begin{equation}
\pi:(z_1,\cdots,z_{n+1})\in S^{2n+1}\mapsto [(z_1,\cdots,z_{n+1})]\in \Bbb{CP}^n(4)
\end{equation}
be the Hopf fibration,
where $|z_1|^2+\cdots+|z_{n+1}|^2=1$ and $\Bbb{CP}^n(4)$ denotes the $n$-dimensional complex projective space with constant holomorphic
sectional curvature $4$, then $M$ is a minimal Legendrian submanifold if and only if $\pi(M)$ is a minimal Lagrangian submanifold of
$\Bbb{CP}^n(4)$ (see e.g. \cite{C-L-U}). Therefore, the rigidity properties of the above 2 classes of submanifolds have essential relationships.
Through the works of Chen-Ogiue\cite{C-O}, Yamaguchi-Kon-Miyahara\cite{Y-K-M} and Li-Li\cite{L-L}, we know:
\begin{thm}
Let $M^n$ be a compact minimal Legendrian submanifold in $S^{2n+1}$. If $|B|^2\leq \f{2}{3}(n+1)$, then
$M$ is either a totally geodesic subsphere ($|B|^2\equiv 0$) or a flat minimal Legendrian torus (here
$n=2$ and $|B|^2\equiv 2$).
\end{thm}
Similarly as in Theorem \ref{Li} and Theorem \ref{Lu4}, this pinching condition is optimal only for the 2-dimensional case.

Among a lot of successive works on this subject (for an imcomplete list, see e.g. \cite{Y2,Y3,B-O,B-O2,O,N-T,U,G,D-V,Xia,L-S}), Luo-Sun-Yin \cite{L-S-Y}
firstly found a pinching condition
which gives a characterization of the Calabi tori:
\begin{thm}
Let $M^n$ be a compact minimal Legendrian submanifold in $S^{2n+1}$, and for each $p\in M$,
$\Theta(p):=\max\limits_{v\in T_p M,|v|=1}|B(v,v)|.$
If $|B|^2\leq \f{n+2}{\sqrt{n}}\Theta$, then $M$ is either a totally geodesic subsphere ($|B|^2\equiv 0$)
or a Calabi torus ($|B|^2\equiv \f{n+2}{\sqrt{n}}\Theta$). Especially if $n=3$, the pinching condition
can be changed weakly to $|B|^2\leq 2+\Th^2$.
\end{thm}
Note that the definition of the Calabi tori will be given in Theorem \ref{Calabi}. These examples firstly appeared in \cite{N}
when H. Naitoh studied isotropic Lagrangian submanifolds in $\Bbb{CP}^n$ with parallel second fundamental form, and were described
from various viewpoints by Castro-Li-Urbano \cite{C-L-U} and Li-Wang \cite{L-W}.

In the present paper, we study the rigidity properties of compact minimal Legendrian submanifolds via the fundamental matrices.
By utilizing Lu's inequality (see Lemma \ref{Lu}), we can establish the following pinching theorem, giving a new characterization
of the Calabi tori:
\begin{main-thm}
Let $M$ be an $n$-dimensional compact minimal Legendrian submanifold in $S^{2n+1}$
and $\la_2$ be second large eigenvalue of the fundamental matrix at each point. If $|B|^2+\la_2\leq n+1$, then $M$ is either a totally geodesic
subsphere (here $|B|^2+\la_2\equiv 0$) or a Calabi torus (here $|B|^2+\la_2\equiv n+1$). This pinching condition is equivalent to
$n^2-n-2\leq R+\mu_2\leq n^2-1$, where $R$ is the scalar curvature of $M$ and $\mu_2$ denotes the second small eigenvalue of
the Ricci curvature tensor.
\end{main-thm}

This means the above pinching condition is optimal for all dimensions. Essentially, this conclusion gives an intrinsic obstruction
for each compact Riemannian manifold becoming a minimal Legendrian submanifold in the unit Euclidean sphere.

Moreover, we establish another rigidity theorem, which shows the pinching condition can be weakened under additional conditions,
e.g. $\la_3\equiv 0$.
\begin{main-thm}\label{m2}
Let $M$ be an $n$-dimensional compact minimal Legendrian submanifold in $S^{2n+1}$ with $n\geq 3$, such that
$\la_3\equiv 0$. If $n\geq 4$, then $M$ has to be totally geodesic. If $n=3$, then $|B|^2\leq \f{16}{3}$ (or $|B|^2\geq \f{16}{3}$)
forces $|B|^2\equiv 0$ or $|B|^2\equiv \f{16}{3}$ (or $|B|^2\equiv \f{16}{3}$). Moreover, $|B|^2\equiv \f{16}{3}$ if and only if
$\pi(M)$ is the equivariant Lagrangian minimal $3$-sphere in $\Bbb{CP}^3(4)$ (see \cite{L-T}), where $\pi$ denotes the Hopf fibration.
\end{main-thm}

This paper will be organized as follows.

In Section \ref{preliminaries}, we introduce the conception of Legendrian submanifolds
in unit Euclidean spheres. From the second fundamental form, we can define a tri-linear symmetric tensor $\si$ and the fundamental
matrix $(S_{ij})$ at each point, whose eigenvalues directly determine the eigenvalues of Ricci curvature tensor, due to the Gauss equation.
Afterwards, the Simons-type identity on the second derivative of $\si$ and Lu's inequality on the upper bound of the squared norm
of Lie brackets of symmetric matrices shall be introduced, which play a crucial role in the following text.

To derive rigidity theorems, it is natural to compute the Laplacian of $\la_1$, i.e. the largest eigenvalue of the fundamental matrix,
but $\la_1$ is not always smooth. To overcome this obstacle, we consider the smooth function
$f_m:=\text{tr}(S^m)=\sum\limits_i \la_i^m$  as in \cite{Lu}. By calculating the Laplacian of $g_m:=f_m^{\f{1}{m}}$ and letting $m\ra \infty$,
we deduce a Simons-type integral inequality (see Proposition \ref{p1}). Based on this inequality, by carefully examining the conditions
when the equality holds, we establish a pinching theorem on $|B|^2+\la_2$, giving a new characterization of the Calabi tori. These are
what we shall do in Section \ref{pinching}.

$\la_3\equiv 0$ is equivalent to saying that the rank of the Gauss map of $M$ is $2$ whenever $|B|^2\neq 0$. Thereby, the Gauss map
is a submersion of $M$ onto a Riemannian surface. By studying the integrability conditions from the viewpoint of complex analysis, we establish a structure theorem for
this type of Legendrian submanifolds (see Theorem \ref{structure}), which immediately implies Main Theorem \ref{m2}. This completes
the whole paper.

	\bigskip\bigskip

    \Section{Preliminaries}{Preliminaries}\label{preliminaries}

    \subsection{Legendrian submanifolds in the unit spheres}

Let $F$ be an isometric immersion from an $n$-dimensional Riemannian manifold $M$ into the
complex Euclidean
$(n+1)$-space $\C^{n+1}:=\{z=(z^1,\cdots,z^{n+1}):z^k\in \C\}$, which is equipped with the canonical complex structure $J$
and the Euclidean inner product $\lan\cdot,\cdot\ran$. If the position vector $F(p)$ of each $p\in M$ always lies
in the unit sphere $S^{2n+1}:=\{z\in \C^{n+1}:\lan z,z\ran=1\}$, $M$ becomes a submanifold of $S^{2n+1}$ and
\begin{equation}
\C^{n+1}=\R F(p)\oplus T_p S^{2n+1}=\R F(p)\oplus T_p M\oplus N_p M
\end{equation}
where $T_p M$ and $N_p M$ are the tangent space and the normal space of $M$ at $p$, respectively. Moreover,
 $M$ is called {\it Legendrian} if and only if
    \begin{equation}
    J(T_p M)\subset N_pM, JF(p)\subset N_pM.
    \end{equation}

Denote by $\n,\ol{\n}$ and $\p$ the Levi-Civita connections on $M, S^{2n+1}$ and $\C^{n+1}$, then
for any tangent vector fields $X,Y$ on $M$,
\begin{equation}
\p_X Y=\ol{\n}_X Y-\lan Y,X\ran F
\end{equation}
and
\begin{equation}
\ol{\n}_X Y=\n_X Y+B(X,Y)
\end{equation}
with $B$ the {\it second fundamental form} of $M$ in $S^{2n+1}$ taking values in the normal bundle $NM$.
In conjunction with $\p_X J=0$, we have
\begin{equation}\aligned
\lan B(X,Y),JF\ran&=\lan \p_X Y,JF\ran=-\lan Y,\p_X(JF)\ran\\
                  &=-\lan Y,JX\ran=0.
\endaligned
\end{equation}
Define
    \begin{equation}
    \si(X,Y,Z):=\lan B(X,Y),JZ\ran,\qquad \forall X,Y,Z\in \G(TM).
    \end{equation}
    then
    $$\aligned
    \si(X,Y,Z)&=\lan \p_X Y,JZ\ran=-\lan Y,\p_X JZ\ran=-\lan Y,J(\p_X Z)\ran\\
    &=\lan \p_X Z,JY\ran=\si(X,Z,Y).
    \endaligned$$
  Along with the symmetry of $B$, we observe that
    $\si$ is a tri-linear symmetric tensor.

    Let $\{E_1,\cdots,E_n\}$ be an arbitrary orthonormal frame field on $M$, then
\begin{equation}
S(X,Y):=\sum_{i,j}\si(E_i,E_j,X)\si(E_i,E_j,Y)
\end{equation}
is a bilinear nonnegative definite symmetric tensor.
Define
\begin{equation}
\left(S_{ij}\right):=\left(S(E_i,E_j)\right)
\end{equation}
to be the {\it fundamental matrices} of $M$ and let
\begin{equation}
\la_1\geq\la_2\geq\cdots \geq\la_n
\end{equation}
be the eigenvalues of $\left(S_{ij}\right)$, then
\begin{equation}
\sum_i \la_i=\text{tr }S=\sum_{i,j,k}\si(E_i,E_j,E_k)\si(E_i,E_j,E_k)=|\si|^2=|B|^2.
\end{equation}

Taking the trace of $B$ gives the {\it mean curvature vector} $H$ of $M$, i.e.
\begin{equation}
H:=\sum_i B(E_i,E_i).
\end{equation}
Now we assume $M$ is a {\it minimal submanifold} of $S^{2n+1}$, i.e. $H\equiv 0$ everywhere on $M$. For each connection $\n$,
let
\begin{equation}
R_{XY}:=-\n_X\n_Y+\n_Y\n_X+\n_{[X,Y]}
\end{equation}
be the associated curvature tensor, then the {\it Gauss equation} says
\begin{equation}\label{Gauss}\aligned
&\lan R_{XY}Z,W\ran=\lan \ol{R}_{XY},Z,W\ran+\lan B(X,Z),B(Y,W)\ran-\lan B(X,W),B(Y,Z)\ran\\
=&\lan X,Z\ran \lan Y,W\ran-\lan X,W\ran\lan Y,Z\ran+\sum_i \si(X,Z,E_i)\si(Y,W,E_i)-\sum_i \si(X,W,E_i)\si(Y,Z,E_i)
\endaligned
\end{equation}
and taking the trace of $R$ implies
\begin{equation}\label{Ric}\aligned
&\text{Ric}(X,Y):=\sum_j \lan R_{XE_j}Y,E_j\ran\\
=&(n-1)\lan X,Y\ran+\sum_{i,j}\si(X,Y,E_i)\si(E_j,E_j,E_i)-\sum_{i,j}\si(X,E_j,E_i)\si(Y,E_j,E_i)\\
=&(n-1)\lan X,Y\ran-S(X,Y).
\endaligned
\end{equation}
i.e. $\la$ is an eigenvalue of the fundamental matrix if and only if $n-1-\la$ is an eigenvalue
of the Ricci curvature tensor of $M$ at the considered point. Again taking the trace of both sides of
(\ref{Ric}), we see the scalar curvature of $M$ equals $n(n-1)-|B|^2$ pointwisely.

The induced normal connection on the normal bundle $NM$ is defined by
\begin{equation}
\n_X \nu=(\ol{\n}_X \nu)^N\qquad X\in \G(TM),\nu\in \G(NM),
\end{equation}
whose corresponding curvature tensor is defined by $R^\bot$.
Then the {\it Ricci equation} says
\begin{equation*}\aligned
&\lan R^\bot_{XY}JZ,JW\ran\\
=&\lan \ol{R}_{XY}JZ,JW\ran+\sum_i \lan B(X,E_i),JZ\ran\lan B(Y,E_i),JW\ran-\sum_i\lan B(X,E_i),JW\ran\lan B(Y,E_i),JZ\ran\\
=&\sum_i \si(X,Z,E_i)\si(Y,W,E_i)-\sum_i \si(X,W,E_i)\si(Y,Z,E_i).
\endaligned
\end{equation*}
Let
\begin{equation}
(\n_X B)(Y,Z):=\n_X (B(Y,Z))-B(\n_X Y,Z)-B(Y,\n_X Z),
\end{equation}
then
\begin{equation}
\aligned
&\lan (\n_X B)(Y,Z),JF\ran\\
=&\lan \n_X (B(Y,Z)),JF\ran=-\n_X \lan B(Y,Z),JF\ran-\lan B(Y,Z),\p_X(JF)\ran\\
=&-\lan B(Y,Z),JX\ran=-\si(X,Y,Z).
\endaligned
\end{equation}
and the {\it Codazzi equation} says
\begin{equation}
(\n_X B)(Y,Z)=(\n_Y B)(X,Z),
\end{equation}
which implies
\begin{equation}
(\n_X \si)(Y,Z,W)=(\n_Y\si)(X,Z,W)=(\n_X \si)(Z,Y,W)=(\n_X \si)(Y,W,Z)
\end{equation}
i.e. $\n\si$ is a four-linear symmetric tensor.

\subsection{The Simons type identity}
The following Simons-type identity play a crucial part in the present paper (see e.g. \cite{L-S-Y}):
\begin{lem}
Assume $M$ is a minimal Legendrian submanifold in $S^{2n+1}$, then
\begin{equation}\label{de1}
\aligned
&\n^2 \si(X,Y,Z):=\sum_i (\n_{E_i}\n_{E_i}\si)(X,Y,Z)\\
=&(n+1)\si(X,Y,Z)-\sum_j[S(X,E_j)\si(E_j,Y,Z)+S(Y,E_j)\si(E_j,Z,X)+S(Z,E_j)\si(E_j,X,Y)]\\
&+2\sum_{i,j,k}\si(X,E_j,E_k)\si(Y,E_k,E_i)\si(Z,E_i,E_j).
\endaligned
\end{equation}
Consequently, let
\begin{equation}
\si_l:=\si(\cdot,\cdot,E_l)
\end{equation}
be the second fundamental form on the direction $JE_l$,
%$\si_l$ be the matrix of $\si(E_t,E_s,E_l)$ with $1\leq t,s\leq n$,
then in terms of matrix notations,
we have
\begin{equation}\label{de3}
\n^2 \si_l=(n+1)\si_l-\sum_j \lan \si_l,\si_j\ran \si_j-\sum_j [\si_j,[\si_j,\si_l]].
\end{equation}
\end{lem}

\begin{proof}
By the Codazzi equation and Ricci identity, a straightforward calculation shows
\begin{equation}\label{de2}\aligned
\n^2 \si(X,Y,Z)&=(\n_{E_i}\n_{E_i}\si)(X,Y,Z)=(\n_{E_i}\n_X\si)(E_i,Y,Z)\\
&=(\n_X\n_{E_i}\si)(E_i,Y,Z)+(R_{X,E_i}\si)(E_i,Y,Z)\\
&=(\n_X\n_Y \si)(E_i,E_i,Z)-\si(R_{X,E_i}E_i,Y,Z)-\si(E_i,R_{X,E_i}Y,Z)-\si(E_i,Y,R_{X,E_i}Z)\\
&=\text{Ric}(X,E_j)\si(E_j,Y,Z)-\lan R_{X,E_i}Y,E_j\ran\si(E_i,E_j,Z)-\lan R_{X,E_i}Z,E_j\ran\si(E_i,Y,E_j)\\
&:=I-II-III.
\endaligned
\end{equation}
(Here and in the sequel we use the summation convention.) According to (\ref{Gauss}) and (\ref{Ric}), we get
\begin{equation}\label{I}\aligned
I=&[(n-1)\lan X,E_j\ran-S(X,E_j)]\si(E_j,Y,Z)\\
=&(n-1)\si(X,Y,Z)-S(X,E_j)\si(E_j,Y,Z),
\endaligned
\end{equation}
\begin{equation}\label{II}\aligned
II=&[\lan X,Y\ran\de_{ij}-\lan X,E_j\ran\lan Y,E_i\ran+\si(X,Y,E_k)\si(E_i,E_j,E_k)-\si(X,E_j,E_k)\si(E_i,Y,E_k)]\si(E_i,E_j,Z)\\
=&-\si(X,Y,Z)+S(Z,E_k)\si(E_k,X,Y)-\si(X,E_j,E_k)\si(Y,E_k,E_i)\si(Z,E_i,E_j)
\endaligned
\end{equation}
and similarly
\begin{equation}\label{III}\aligned
III=-\si(X,Y,Z)+S(Y,E_k)\si(E_k,Z,X)-\si(X,E_k,E_j)\si(Y,E_j,E_i)\si(Z,E_i,E_k).
\endaligned
\end{equation}
Substituting (\ref{I})-(\ref{III}) into (\ref{de2}) gives (\ref{de1}). Letting $X:=E_t,Y:=E_s,Z:=E_l$ in (\ref{de1}), we can derive
\begin{equation}\label{de4}
\n^2 \si_l=(n+1)\si_l-(S\si_l+\si_lS+\sum_j \lan \si_l,\si_j\ran \si_j)+2\sum_k \si_k\si_l\si_k,
\end{equation}
where
\begin{equation}\label{S1}\aligned
\lan \si_l,\si_j\ran:=&\sum_{t,s}\si_l(E_t,E_s)\si_j(E_t,E_s)\\
=&\sum_{t,s}\si(E_t,E_s,E_l)\si(E_t,E_s,E_j)=S_{lj}.
\endaligned
\end{equation}
Finally, (\ref{de3}) immediately follows from (\ref{de4}),
\begin{equation}
[\si_k,[\si_k,\si_l]]=\si_k\si_k\si_l-2\si_k\si_l\si_k+\si_l\si_k\si_k
\end{equation}
and
\begin{equation}\aligned
(\si_k\si_k)_{ts}=&\si_k(E_t,E_i)\si_k(E_i,E_s)\\
=&\si(E_t,E_i,E_k)\si(E_i,E_s,E_k)\\
=&S_{ts}.
\endaligned
\end{equation}

\end{proof}
	
	\subsection{On Lu's inequality}

For $2$ real $(n\times n)$-matrices $A=(a_{ij})$ and $B=(b_{ij})$, let
\begin{equation}
\lan A,B\ran:=\sum_{i,j=1}^n a_{ij}b_{ij}=\text{tr }(AB^T)
\end{equation}
with $(\cdot)^T$ denoting the transpose of a matrix, which induces the Hilbert-Schmidt norm:
\begin{equation}
\|A\|=\sqrt{\lan A,A\ran}=\sqrt{\sum_{i,j=1}^n a_{ij}^2}.
\end{equation}

Z. Q. Lu \cite{Lu} established the following matrix inequality, which is the main algebraic tool of the present paper.
\begin{lem}\label{Lu}
Let $A_1,A_2,\cdots,A_n$ be $(n\times n)$-symmetric matrices, such that
\begin{itemize}
\item $\lan A_\a,A_\be\ran=0$ whenever $\a\neq \be$;
\item $\|A_1\|=1$;
\item $\|A_2\|\geq \cdots\geq \|A_m\|$.
\end{itemize}
Then
\begin{equation}
\sum_{\a=2}^n\left\|[A_1,A_\a]\right\|^2\leq \|A_2\|^2+\sum_{\a=2}^{n}\|A_\a\|^2
\end{equation}
and the equality holds if and only if, after an orthonormal base change and up to a sign, we have
$A_{k+2}=\cdots=A_n=0$,
\begin{equation}
A_1=\la\left(\begin{matrix} k & & \\ & -I_k &\\ & & O\end{matrix}\right)
\end{equation}
and $A_\a$ ($2\leq \a\leq k+1$) is $\mu$ times the matrix whose only nonzero entries are $1$ at the $(1,\a)$ and $(\a,1)$ places,
i.e. $A_\a=E_{1,\a}+E_{\a,1}$. Here $1\leq k\leq n-1$, $\la=\f{1}{\sqrt{k(k+1)}}$ and $\mu$ is a constant.  %$\mu\in [-\f{1}{\sqrt{2}},\f{1}{\sqrt{2}}]$.

\end{lem}
	
\begin{rem}
%As pointed out by the first author of the present paper in ,
 %the original version of Lu's result (Lemma 2 of \cite{Lu})
The above conclusion is just Lemma 2.2 of \cite{W}, which is the revised version of Lemma 2 of \cite{Lu}.
Here the author found there are more cases when the Lu's equality holds and gave another proof by using Lagrange Muliplier method.

\end{rem}

	\bigskip\bigskip

\Section{An optimal pinching theorem on $|B|^2+\la_2$}{An optimal pinching theorem on $|B|^2+\la_2$}\label{pinching}
\subsection{A Simons type integral inequality}
For any positive number $m$, we consider the $C^\infty$-function
\begin{equation}
f_m:=\text{tr}(S^m)
\end{equation}
as in \cite{Lu}. Let $p\in M$ be an arbitrary point and $\{E_1,\cdots,E_n\}$ be a local orthonormal frame field,
such that the fundamental matrix at $p$ is diagonalized, i.e.
\begin{equation}
S_{ij}(p)=\la_i\de_{ij}
\end{equation}
and
\begin{equation}
\la_1=\cdots=\la_r>\la_{r+1}\geq \cdots\geq \la_n.
\end{equation}

Combining (\ref{S1}) and (\ref{de3}) implies
\begin{equation}\label{de5}\aligned
\f{1}{2}\n^2 S_{ll}=&\f{1}{2}\n^2 \lan \si_l,\si_l\ran=\lan \n^2 \si_l,\si_l\ran+\lan \n \si_l,\n \si_l\ran\\
=&(n+1)S_{ll}-(S^2)_{ll}-\sum_j \|[\si_j,\si_l]\|^2+\lan \n \si_l,\n \si_l\ran\\
=&(n+1)\la_l-\la_l^2-\sum_j \|[\si_l,\si_j]\|^2+\lan \n \si_l,\n \si_l\ran.
\endaligned
\end{equation}
By Lemma \ref{Lu}, for each $1\leq l\leq r$,
\begin{equation}\label{Lu2}\aligned
&\sum_j \|[\si_l,\si_j]\|^2
\leq \|\si_l\|^2\left(\sum_{j\neq l} \|\si_j\|^2+\|\si_2\|^2\right)\\
=&\la_l(\sum_{j\neq l} \la_j+\la_2),
\endaligned
\end{equation}
and for each $r+1\leq l\leq n$,
\begin{equation}\label{Lu3}\aligned
&\sum_j \|[\si_l,\si_j]\|^2
\leq \|\si_l\|^2\left(\sum_{j\neq l} \|\si_j\|^2+\|\si_1\|^2\right)\\
=&\la_l(\sum_{j\neq l} \la_j+\la_1).
\endaligned
\end{equation}

Note that
\begin{equation}
f_m=\text{tr}(S^m)=\sum_{i_1,\cdots,i_m}S_{i_1i_2}S_{i_2i_3}\cdots S_{i_mi_1},
\end{equation}
we have
\begin{equation}\label{f1}\aligned
\De f_m=&\sum_{i_1,\cdots,i_m}\n^2 S_{i_1i_2}\cdot S_{i_2i_3}\cdots S_{i_mi_1}+\cdots+\sum_{i_1,\cdots,i_m}S_{i_1i_2}S_{i_2i_3}\cdots \n^2 S_{i_mi_1}\\
&+\sum_{i_1,\cdots,i_m}\sum_{j<k}S_{i_1i_2}\cdots\widehat{S_{i_ji_{j+1}}}\cdots\widehat{S_{i_ki_{k+1}}}\cdots S_{i_mi_1}\lan \n S_{i_ji_{j+1}},\n S_{i_ki_{k+1}}\ran\\
=&m\sum_l \n^2 S_{ll}\cdot \la_l^{m-1}+m\sum_{l<p}\sum_{s+t=m-2}|\n S_{lp}|^2\la_l^s \la_p^t+m(m-1)\sum_l |\n S_{ll}|^2\la_l^{m-2}\\
\geq &2m(n+1)f_m-2m\sum_{1\leq l\leq r}\la_l^m(\sum_j \la_j+\la_2)-2m\sum_{r+1\leq l\leq n}\la_l^m(\sum_j \la_j+\la_1)\\
&+2m\sum_l \lan \n \si_l,\n \si_l\ran\la_l^{m-1}+m(m-1)\sum_l |\n S_{ll}|^2\la_l^{m-2}
\endaligned
\end{equation}
by using (\ref{de5}), (\ref{Lu2}) and (\ref{Lu3}).
On the other hand, with the aid of the Cauchy inequality, we get
\begin{equation}\label{f2}\aligned
|\n f_m|^2=&m^2 \sum_i\left(\sum_l (\n_{E_i}S_{ll})\la_l^{m-1}\right)^2\\
=&m^2 \sum_i\left(\sum_l (\n_{E_i}S_{ll})\la_l^{\f{m}{2}-1}\cdot \la_l^{\f{m}{2}}\right)^2\\
\leq & m^2 f_m \sum_{l}|\n S_{ll}|^2 \la_l^{m-2}.
\endaligned
\end{equation}

Let
\begin{equation}
g_m:=(f_m)^{\f{1}{m}}.
\end{equation}
then (\ref{f1}) and (\ref{f2}) implies
\begin{equation}\aligned
\De g_m=&\f{1}{m}(f_m)^{\f{1}{m}-1}\De f_m+\f{1}{m}\left(\f{1}{m}-1\right)(f_m)^{\f{1}{m}-2}|\n f_m|^2\\
\geq & 2g_m\big[n+1-f_m^{-1}\sum_{1\leq l\leq r}\la_l^m(\sum_j \la_j+\la_2)-f_m^{-1}\sum_{r+1\leq l\leq n}\la_l^m(\sum_j \la_j+\la_1)\\
&+f_m^{-1}\sum_l \lan \n\si_l,\n\si_l\ran \la_l^{m-1} \big]
\endaligned
\end{equation}
whenever $f_m(p)\neq 0$. Noting that
\begin{equation}
\lim_{m\ra \infty}f_m^{-1}\la_l^m=\lim_{m\ra\infty}\f{(\f{\la_l}{\la_1})^m}{\sum\limits_j (\f{\la_j}{\la_1})^m}=\left\{
\begin{matrix}\f{1}{r} & l\leq r\\ 0 & l\geq r+1\end{matrix}\right.
\end{equation}
we have
\begin{equation}
\lim_{m\ra \infty}\De g_m\geq 2\la_1(n+1-\sum_j \la_j-\la_2)+\f{2}{r}\sum_{1\leq l\leq r}\lan \n \si_l,\n \si_l\ran.
\end{equation}
In conjunction with $\int_M \De g_m=0$, we get the following Simons type integral inequality:
\begin{pro}\label{p1}
Let $M$ be an $n$-dimensional compact minimal Legendrian submanifold in $S^{2n+1}$, %$(S_{ij})$ be its fundamental matrices
 $\la_1=\cdots=\la_r>\la_{r+1}\geq \cdots\geq \la_n$ be the eigenvalues of the fundamental matrix at each considered point, then
\begin{equation}
\int_M \la_1(n+1-|B|^2-\la_2)+\f{1}{r}\sum_{1\leq l\leq r}\lan \n \si_l,\n \si_l\ran*1\leq 0.
\end{equation}
\end{pro}

\subsection{A characterization of the Calabi tori}

If $|B|^2+\la_2\equiv n+1$, then Proposition \ref{p1} and Lemma \ref{Lu} tell us $\n \si_l=0$ for all $1\leq l\leq r$, and there
exist $1\leq k\leq n-1$ and an orthogonal matrix $T$, such that
\begin{equation}
\si_1=\la T^t\left(\begin{matrix} k & & \\ & -I_k &\\ & & O\end{matrix}\right)T,
\end{equation}
$\si_l=\mu T^t(E_{1l}+E_{l1})T$ for each $2\leq l\leq k+1$, and $\si_{k+2}=\cdots=\si_n=0$.

Let $\si_{lpq}:=\si(E_l,E_p,E_q)$. Due to the symmetry of $\si$, for each $p\geq k+2$ and $2\leq l\leq k+1$,
\begin{equation*}
0=\si_{plp}=\si_{lpp}=(\si_l)_{pp}=\mu\big((T^t)_{p1}T_{lp}+(T^t)_{pl}T_{1p}\big)=2\mu T_{1p}T_{lp}
\end{equation*}
implies
\begin{equation*}
T_{1p}T_{lp}=0.
\end{equation*}
In conjunction with
$$0=\si_{1pp}=\la\big[(T^t)_{p1}kT_{1p}-\sum_{2\leq l\leq k+1}(T^t)_{pl}T_{lp}\big]=\la(kT_{1p}^2-\sum_{2\leq l\leq k+1}T_{lp}^2),$$
we have
\begin{equation}
T_{lp}=0\qquad (\forall 1\leq l\leq k+1,p\geq k+2).
\end{equation}
It follows that
\begin{equation}
\de_{ij}=(T^t T)_{ij}=\sum_{1\leq l\leq k+1}T_{li}T_{lj}\qquad (\forall 1\leq i,j\leq k+1)
\end{equation}
and hence
\begin{equation}\aligned
(\si_1)_{ij}=&\la(kT_{1i}T_{1j}-\sum_{2\leq l\leq k+1}T_{li}T_{lj})\\
=&\la(-\de_{ij}+(k+1)T_{1i}T_{1j}).
\endaligned
\end{equation}
Thereby, due to $(\si_1)_{1i}=(\si_i)_{11}$, $(\si_1)_{ij}=(\si_i)_{1j}=(\si_j)_{1i}$, $(\si_1)_{ii}=(\si_i)_{1i}$
and $(\si_i)_{jj}=(\si_j)_{ij}$ for each distinct $i,j$ lying between $2$ and $k+1$, we obtain several constraints on $T$ as follows:
\begin{eqnarray}
\la(k+1)T_{11}T_{1i}&=&2\mu T_{11}T_{i1},\label{T1}\\
\la(k+1)T_{1i}T_{1j}&=&\mu(T_{11}T_{ij}+T_{i1}T_{1j})=\mu(T_{11}T_{ji}+T_{j1}T_{1i}),\label{T2}\\
\la(-1+(k+1)T_{1i}^2)&=&\mu(T_{11}T_{ii}+T_{i1}T_{1i}),\label{T3}\\
2T_{1j}T_{ij}&=&T_{1i}T_{jj}+T_{ji}T_{1j}.\label{T4}
\end{eqnarray}
If $T_{11}=0$, then there exists $2\leq j\leq k+1$, such that $T_{1j}\neq 0$, then (\ref{T2}) implies
$\la(k+1)T_{1i}=\mu T_{i1}$ for each $i\neq 1,j$. Substituting it into (\ref{T3}) gives
$\la(-1+(k+1)T_{1i}^2)=\la(k+1)T_{1i}^2$. This is a contradiction. Therefore $T_{11}\neq 0$
and then (\ref{T1}) implies $\la(k+1)T_{1i}=2\mu T_{i1}$. Thus $T_{i1}T_{1j}=T_{j1}T_{1i}$ and
(\ref{T2}) gives $T_{ij}=T_{ji}$. It follows that $T_{1i}^2=1-\sum\limits_{2\leq j\leq k+1}T_{ji}^2=1-\sum\limits_{2\leq j\leq k+1}T_{ij}^2=T_{i1}^2$.
Now we claim $T_{1i}=0$ for each $2\leq i\leq k+1$. Otherwise, there exists $l$ such that $T_{1l}\neq 0$ and hence
 $\la(k+1)=\pm 2\mu$. If $\la(k+1)=2\mu$, (\ref{T1})-(\ref{T4}) yield
 \begin{eqnarray}
T_{1i}&=&T_{i1},\label{T5}\\
T_{11}T_{ij}&=&T_{1i}T_{1j},\label{T6}\\
-\la&=&\mu(T_{11}T_{ii}-T_{1i}^2),\label{T7}\\
T_{1i}T_{jj}&=&T_{1j}^2.\label{T8}
\end{eqnarray}
Multiplying both sides of (\ref{T6}) and (\ref{T8}) implies $T_{ij}T_{1i}(T_{11}T_{jj}-T_{1j}^2)=0$. If there exists $1\leq i<j\leq k+1$,
such that $T_{ij}=0$, then (\ref{T6}) forces $T_{1i}\neq 0$ and hence $T_{11}T_{jj}-T_{1j}^2=0$, causing a contradiction to (\ref{T7}).
Therefore $T_{ij}=0$ for each $1\leq i<j\leq k+1$. Due to (\ref{T6}) and (\ref{T8}), for each $j\neq 1,l$, $T_{1j}=T_{jj}=0$,
which also causes a contradiction to (\ref{T7}). On the other hand, if $\la(k+1)=-2\mu$, we can proceed similarly as above to obtain
contradictions. Thereby, substituting $T_{1i}=0$ into (\ref{T2}) gives $T_{ij}=0$ for each $2\leq i<j\leq k+1$. This means $T$
is a diagonal matrix, and further calculation shows
\begin{equation}
\si_1=\la \left(\begin{matrix} k & & \\ & -I_k &\\ & & O\end{matrix}\right),\quad \si_l=-\la (E_{1l}+E_{l1})\ (2\leq l\leq k+1)
\end{equation}
and $\si_{k+2}=\cdots=\si_n=0$. In other words,
\begin{equation}
\si_{111}=k\la,\quad \si_{1ll}=-\la\ (2\leq l\leq k+1)
\end{equation}
and the others are $0$.

If $k+1<n$, then for any $1\leq l\leq k+1$ and $p\geq k+2$, differentiating both sides of $\si(E_1,E_l,E_p)=0$ gives
\begin{equation*}
0=\n_X \si(E_1,E_l,E_p)=\lan \n_X E_p,E_l\ran \si_{1ll}
\end{equation*}
(where we have used $\n \si_1=0$)
and hence
\begin{equation}
\lan\n_X E_p,E_l\ran=\lan \n_X E_l, E_p\ran=0.
\end{equation}
Consequently
\begin{equation}
\lan R_{XY}E_l,E_p\ran=0.
\end{equation}
On the other hand, the Gauss equation (\ref{Gauss}) implies
\begin{equation}
\lan R_{E_lE_p}E_l,E_p\ran=1,
\end{equation}
causing a contradiction.

Therefore $k+1=n$,
\begin{equation}
\la_1=\la^2(n-1)n,\quad \la_2=\cdots=\la_n=2\la^2,
\end{equation}
and $\sum_j \la_j+\la_2\equiv n+1$ means
\begin{equation}
\la=\sqrt{\f{1}{n}}.
\end{equation}
Differentiating both sides of $\si_{11l}=0$ gives
\begin{equation}
0=2\lan \n_X E_1,E_l\ran\si_{1ll}+\lan \n_X E_l,E_1\ran\si_{111}=(n+1)\la \lan \n_X E_l,E_1\ran
\end{equation}
for each $2\leq l\leq n$. This means $\mc{D}:=\text{span}\{E_2,\cdots,E_n\}$ is an integral distribution. Let
$N$ be the integral submanifold of $\mc{D}$. Noting that
\begin{equation}
\p_{E_i}(F\w E_2\w\cdots\w E_n\w JE_1)=0
\end{equation}
and
\begin{equation}
\lan B(E_i,E_j),JE_1\ran=-\sqrt{\f{1}{n}}\de_{ij},
\end{equation}
we know $N$ is a subsphere of $S^{n}$. On the other hand, let $\g$ be the integral curve of $E_1$, then
\begin{equation}\aligned
\p_{E_1}F=&E_1,\\
\p_{E_1}(JF)=&JE_1,\\
\p_{E_1}E_1=&-F+(n-1)\sqrt{\f{1}{n}}JE_1,\\
\p_{E_1}(JE_1)=&-JF-(n-1)\sqrt{\f{1}{n}}E_1.
\endaligned
\end{equation}
Noting that $\p_{E_1}(F\w JF\w E_1\w JE_1)=0$, $\g$ is a Legendrian curve in $S^3$. More precisely,
\begin{equation}\label{curve}
\g(t)=(\g_1(t),\g_2(t):=\left(\sqrt{\f{n}{n+1}}\exp(\sqrt{-1}\sqrt{\f{1}{n}}t),\sqrt{\f{1}{n+1}}\exp(-\sqrt{-1}\sqrt{n}t)\right)\qquad t\in S^1.
\end{equation}
In summary, we establish a pinching theorem on $|B|^2+\la_2$ as follows:

\begin{thm}\label{Calabi}
Let $M$ be an $n$-dimensional compact minimal Legendrian submanifold in $S^{2n+1}$, %$(S_{ij})$ be its fundamental matrices
 $\la_2$ be second large eigenvalue of the fundamental matrix at each considered point. If $|B|^2+\la_2\leq n+1$, then $M$ is either a totally geodesic
subsphere (here $|B|^2+\la_2\equiv 0$) or a Calabi torus (here $|B|^2+\la_2\equiv n+1$). More precisely, let
$M:=S^{n-1}\times S^1$, $\phi$ be canonical embedding of $S^{n-1}$ into $\R^{n}$ and $\g(t)=(\g_1(t),\g_2(t))$ be a Legendiran curve
given in (\ref{curve}), then $F: (x,t)\in S^{n-1}\times S^1\mapsto (\g_1(t)\phi(x),\g_2(t))\in \C^{n+1}$ defines
a compact Legendrian submanifold in $S^{2n+1}$, called a Calabi torus.

\end{thm}

According to (\ref{Ric}), we can rewrite the above theorem as an intrinsic rigidity conclusion:

\begin{cor}
Let $M$ be an $n$-dimensional compact minimal Legendrian submanifold in $S^{2n+1}$, $R$ be the scalar curvature of $M$
and $\mu_1\leq \mu_2\leq \cdots\leq \mu_n$ be eigenvalue of the Ricci curvature tensor. If $n^2-n-2\leq R+\mu_2\leq n^2-1$, then $M$ is either a totally geodesic
subsphere (here $R+\mu_2\equiv n^2-1$) or a Calabi torus (here $R+\mu_2\equiv n^2-n-2$).
\end{cor}
	
	\bigskip\bigskip
	\Section{Compact minimal Legendrian submanifolds with $\la_3\equiv 0$}{Compact minimal Legendrian submanifolds with $\la_3\equiv 0$}

As showing in \S \ref{pinching}, For an $n$-dimensional compact minimal Legendrian submanifold $M\subset S^{2n+1}$,
if $|B|^2+\la_2\leq n+1$ and $\la_n\equiv 0$, then $M$ has to be totally geodesic. It is natural to ask whether we can
find a larger number $C(k)$ with $k\leq n-1$, such that the pinching condition $|B|^2+\la_2\leq C(k)$ and $\la_{k+1}\equiv 0$
implies $|B|^2+\la_2\equiv 0$ or $|B|^2+\la_2\equiv C(k)$. In this section, we consider the simplest case of $k=2$.
(Since $\la_2\equiv 0$, $H\equiv 0$ and the symmetry of $\si$ immediately force $\si\equiv 0$.)

\subsection{The rank of Gauss maps}

Let $M^n\subset S^{2n+1}$ be a Legendrian submanifold and $F:M\ra \C^{n+1}$ be the position vector. Then
$\g: M\ra \grs{n+1}{n+1}$
\begin{equation}
\g(p)=N_p M=JT_p M\oplus \R JF(p)
\end{equation}
is the {\it Gauss map} of $M$ via parallel translation in $\C^{n+1}$, where $\grs{n+1}{n+1}$ is the {\it Grassmannian manifold}
consisting of all oriented $(n+1)$-dimensional subspace of $\R^{2n+2}=\C^{n+1}$. Using Pl\"{u}cker coordinates,
the Gauss map can be written as
\begin{equation}
\g=JE_1\w \cdots \w JE_n\w JF,
\end{equation}
where $\{E_1,\cdots,E_n\}$ is an orthonormal frame field on $M$. Thus
\begin{equation}\aligned
\g_* X=&\sum_l JE_1\w\cdots\w JE_{l-1}\w \p_X(JE_l)\w JE_{l+1}\w\cdots\w JE_n\w JF+JE_1\w\cdots\w JE_n\w \p_X(JF)\\
=&- \sum_{l,p}JE_1\w\cdots\w JE_{l-1}\w \si(X,E_l,E_p)E_p\w JE_{l+1}\w\cdots\w JE_n\w JF.
\endaligned
\end{equation}
This means $\g_* X=0$ if and only if $\si(X,\cdot,\cdot)=0$ and we obtain the following conclusion:

\begin{pro}
For an $n$-dimensional Legendrian submanifold $M$ in $S^{2n+1}$, let $(S_{ij})$ be the fundamental matrix at $p\in M$,
whose eigenvalues are $\la_1\geq \cdots \geq \la_k>0=\la_{k+1}=\cdots=\la_n$, then the rank of Gauss map $\g$ at $p$ equals $k$.
\end{pro}

\subsection{A structure theorem for $3$-dimensional minimal Legendrian submanifolds}
	
Now we consider compact minimal Legendrian submanifolds with $\la_3\equiv 0$. As shown in \cite{D-K-S-T}, $M$ has to be totally geodesic
whenever $n\geq 4$. So we only consider the case of $n=3$.

Let
\begin{equation}
M^+:=\{p\in M:|B|^2(p)\neq 0\},
\end{equation}
then due to the analyticity, $M^+$ is either empty or an open and dense subset of $M$; The former case means $M$ is totally geodesic, so we just
consider the latter one in the following text.

Let $\g: M^+\ra \grs{4}{4}$
be the Gauss map, which has constant rank $2$. Hence the image of $\g$ is an immersed $2$-dimensional submanifold of $\grs{4}{4}$,
and each connected component of any level set of $\g$ has to be a curve, which is called a G-loop. Along an arbitrary G-loop $\xi$,
let $T$ be the unit tangent vector field, and
\begin{equation}
\mc{D}:=\text{span}\{X,T\},\qquad \mc{D}^\bot:=\{v\in TM: \lan v,T\ran=0\},
\end{equation}
then it follows that (see \cite{D-K-S-T}):
\begin{itemize}
\item $B(T,v)=0$ for every $v\in TM$;
\item $\xi$ is a geodesic of $S^7$, i.e. $X$ and $T$ span a fixed subspace $\mc{D}$ of $\R^8$;
\item $\mc{D}^\bot, J(\mc{D}), J(\mc{D}^\bot)$ are all parallel along $\xi$.
\end{itemize}

We call $p\sim q$ whenever $p$ and $q$ lies in the same G-loop and denote by
\begin{equation}
\Si:=M^+\backslash \sim=\{[p]:p\in M^+\}
\end{equation}
the loop space equipped with the quotient topology. Locally,
\begin{equation}
[\g]:[p]\in \Si\mapsto \g(p)
\end{equation}
is a one-to-one correspondence between
a sufficiently small open subset of $\Si$ and the corresponding open subset of the Gauss image of $M^+$. Thus, $\Si$ can be seen as
a Riemannian surface, so that $[\g]$ is holomorphic. In other word, the complex structure $J_0$ on $\Si$ satisfies
\begin{equation}\label{com_str0}
J_0(\pi_*(E_1))=\pi_*(E_2),
\end{equation}
where $\pi:p\in M\ra [p]\in \Si$ and $\{E_1,E_2\}$ is an orientable orthonormal basis of $\mc{D}^\bot$.

Let
\begin{equation}
W:=\f{\sqrt{2}}{2}(E_1-\sqrt{-1}E_2)
\end{equation}
be a $(1,0)$-vector in $\mc{D}^\bot\otimes \C$, then
\begin{equation}
\lan W,W\ran=\lan \ol{W},\ol{W}\ran=0,\quad \lan W,\ol{W}\ran=1,
\end{equation}
\begin{equation}\label{b1}
B(W,\ol{W})=\f{1}{2}(B(E_1,E_1)+B(E_2,E_2))=0
\end{equation}
and
\begin{equation}
|B(W,W)|^2=|B(E_1,E_1)-\sqrt{-1}B(E_1,E_2)|^2=\f{|B|^2}{2}>0.
\end{equation}

On the other hand, $W$ can be seen as a complex vector-valued function on $\Si$, since it is parallel along each
G-loop. Let $z$ be a local complex coordinate of $\Si$, we shall calculate $W_z$ and $W_{\bar{z}}$.

From (\ref{com_str0}) we know $\pi_*W=f \pf{z}$ with a complex function $f$. Since
$$\lan \partial_W W,F\ran=-\lan W,\partial_W F\ran=\lan W,W\ran=0$$
we have
\begin{equation}\lan W_z,X\ran=0.
\end{equation}
Differentiating both sides of $B(T,\ol{W})=0$ shows
$$\aligned
0&=\n_W B(T,\ol{W})=(\n_W B)(T,\ol{W})+B(\n_W T,\ol{W})\\
&=(\n_T B)(W,\ol{W})+\lan \n_W T,W\ran B(\ol{W},\ol{W})+\lan \n_W T,\ol{W}\ran B(W,\ol{W})\\
&=\lan \n_W T,W\ran B(\ol{W},\ol{W}).
\endaligned$$
This implies $\lan \p_W W,T\ran=\lan \n_W W,T\ran=-\lan \n_W T,W\ran=0$ and hence
\begin{equation}
\lan W_z,T\ran=0.
\end{equation}
Since
$$\aligned
\lan \p_W W,JF\ran&=-\lan W,\p_W(JF)\ran=-\lan W,J\p_W F\ran=-\lan W,JW\ran=0,\\
\lan \p_W W,JT\ran&=\lan B(W,W),JT\ran=\lan B(T,W),JW\ran=0,
\endaligned$$
and
$$\lan \p_W W,J\ol{W}\ran=\lan B(W,W),J\ol{W}\ran=\lan B(W,\ol{W}),JW\ran=0,$$
we have
\begin{equation}
\lan W_z,JX\ran=0,\quad \lan W_z,JT\ran=0,\quad \lan W_z,JW\ran=0.
\end{equation}
In conjunction with
\begin{equation}
\lan W_z,W\ran=\f{1}{2}\lan W,W\ran_z=0,
\end{equation}
we can write
\begin{equation}\label{wz1}
W_z=h W+\mu J\ol{W}
\end{equation}
with
\begin{equation}
h:=\lan W_z, \ol{W}\ran, \quad \mu:=\lan W_z,JW\ran.
\end{equation}
Note that replacing $W$ by $\hat{W}:=e^{i\th}W$ with suitable $\th$ makes sure $\hat{\mu}:=\lan \hat{W}_z,J\hat{W}\ran=e^{2i\th}\lan W_z,JW\ran\in \R^+$. Thereby, we can assume $\mu$ takes positive real values everywhere.

Combining with (\ref{b1}) and $\pi_* \ol{W}=\bar{f}\pf{\bar{z}}$ implies $W_{\bar{z}}$ is a vector field in $(\mc{D}\oplus \mc{D}^\bot)\otimes \C$. Moreover,
\begin{equation}
\lan W_{\bar{z}},W\ran=\f{1}{2}\lan W,W\ran_{\bar{z}}=0,
\end{equation}
and
\begin{equation}
\lan W_{\bar{z}},\ol{W}\ran=-\lan W,\ol{W}_{\bar{z}}\ran=-\bar{h}
\end{equation}
imply the existence of $U$ in $\mc{D}\otimes \C$, such that
\begin{equation}\label{wz2}
W_{\bar{z}}=-\bar{h}W+U.
\end{equation}

A straightforward calculation based on (\ref{wz1}) and (\ref{wz2}) shows
\begin{equation}
\aligned
W_{z\bar{z}}&=h_{\bar{z}}W+hW_{\bar{z}}+\mu_{\bar{z}}J\ol{W}+\mu J\ol{W}_{\bar{z}}\\
&=(h_{\bar{z}}-\mu^2-|h|^2)W+(\mu_{\bar{z}}+\mu\bar{h})J\ol{W}+h U
\endaligned
\end{equation}
and
\begin{equation}
\aligned
W_{\bar{z}z}&=-\bar{h}_z W-\bar{h}W_z+U_z\\
&=(-\bar{h}_z-|h|^2)W-\mu\bar{h}J\ol{W}+U_z.
\endaligned
\end{equation}
In conjunction with
\begin{equation}
\lan U_z, JW\ran=-\lan U,JW_z\ran=0,
\end{equation}
we have
\begin{equation}\label{u1}
\mu_{\bar{z}}+2\mu\bar{h}=0
\end{equation}
and
\begin{equation}\label{u3}
U_z=(h_{\bar{z}}+\bar{h}_z-\mu^2)W+hU,
\end{equation}
which imply
\begin{equation}
(\log\mu)_{\bar{z}}=-2\bar{h},
\end{equation}
\begin{equation}\label{u4}
\aligned
\la:&=\lan U,\ol{U}\ran=\lan W_{\bar{z}},\ol{U}\ran=-\lan W,\ol{U}_{\bar{z}}\ran\\
&=-\bar{h}_z-h_{\bar{z}}+\mu^2=(\log\mu)_{z\bar{z}}+\mu^2
\endaligned
\end{equation}
and
\begin{equation}\label{u2}
\lan U,U\ran_z=2\lan U_z,U\ran=2h\lan U,U\ran.
\end{equation}
Combining (\ref{u1}) and (\ref{u2}) gives
\begin{equation}\label{con1}
\left(\mu\lan \ol{U},\ol{U}\ran\right)_{\bar{z}}=-2\mu\bar{h}\lan \ol{U},\ol{U}\ran+2\mu \bar{h}\lan \ol{U},\ol{U}\ran=0.
\end{equation}
Note that
\begin{equation}
\Psi:=\mu \lan \ol{U},\ol{U}\ran dz^3=\lan W_z,JW\ran\lan \ol{W}_z,\ol{W}_z\ran dz^3
\end{equation}
is independent of the choice of $W$ and $z$, hence (\ref{con1}) means $\Psi$ is a holomorphic $3$-form globally defined on $\Si$.

If $\Psi$ does not vanish everywhere, then we can find an open subset of $\Si$ which admits a complex coordinate $w$, such that
$\Psi=dw^3=\lan W_w,JW\ran\lan \ol{W}_w,\ol{W}_w\ran dw^3$. Therefore, without loss of generality we can assume
\begin{equation}\label{u5}
\mu \lan \ol{U},\ol{U}\ran=1
\end{equation}
holds locally. By computing,
$$\aligned
\lan U_{\bar{z}},W\ran&=-\lan U,W_{\bar{z}}\ran=-\lan U,U\ran=-\mu^{-1},\\
\lan U_{\bar{z}},\ol{W}\ran&=\lan U,\ol{W}_{\bar{z}}\ran=0,\\
\lan U_{z\bar{z}},W\ran&=\lan U_z,W\ran_{\bar{z}}-\lan U_z,W_{\bar{z}}\ran=-h\mu^{-1},\\
\lan U_{\bar{z}z},W\ran&=\lan U_{\bar{z}},W\ran_z-\lan U_{\bar{z}},W_z\ran=-(\mu^{-1})_z-h\mu^{-1}.
\endaligned$$
This shows $\mu_z=0$, i.e. $\mu$ is constant, then (\ref{u1}), (\ref{u3}) and (\ref{u4}) implies
\begin{equation}
h\equiv 0,\qquad U_z=-\mu^2 W
\end{equation}
 and
\begin{equation}\label{u6}
\lan U,\ol{U}\ran=\mu^2.
\end{equation}
Differentiating both sides of (\ref{u5}) and (\ref{u6}), we can derive $\lan U_{\bar{z}},U\ran=\lan U_{\bar{z}},\ol{U}\ran=0$, hence
\begin{equation}
U_{\bar{z}}=-\mu^{-1}\ol{W}.
\end{equation}
Comparing $U_{z\bar{z}}=-\mu^2 W_{\bar{z}}=-\mu^2 U$ and $U_{\bar{z}z}=-\mu^{-1}\ol{W}_z=-\mu^{-1}\ol{U}$ gives
$\mu\equiv 0$, which causes a contradiction to $\mu>0$.

Therefore $\Psi\equiv 0$, i.e. $\lan U,U\ran\equiv 0$. Since
$$\aligned
\lan U_{\bar{z}},W\ran&=-\lan U,W_{\bar{z}}\ran=-\lan U,U\ran=0,\\
\lan U_{\bar{z}},\ol{W}\ran&=\lan U,\ol{W}_{\bar{z}}\ran=0,\\
\lan U_{\bar{z}},U\ran&=\f{1}{2}\lan U,U\ran_{\bar{z}}=0,\\
\lan U_{\bar{z}},\ol{U}\ran&=\lan U,\ol{U}\ran_{\bar{z}}-\lan U,\ol{U}_{\bar{z}}\ran=\la_{\bar{z}}-\la \bar{h},
\endaligned$$
we have
\begin{equation}
U_{\bar{z}}=\left((\log\la)_{\bar{z}}-\bar{h}\right)U.
\end{equation}
Comparing
\begin{equation*}
U_{z\bar{z}}=(-\la W+hU)_{\bar{z}}=(-\la_{\bar{z}}+\la \bar{h})W+\left(h_{\bar{z}}-\la+h(\log\la)_{\bar{z}}-|h|^2\right)U
\end{equation*}
with
\begin{equation*}
U_{\bar{z}z}=\left[\left((\log\la)_{\bar{z}}-\bar{h}\right)U\right]_z
=(-\la_{\bar{z}}+\la \bar{h})W+\left((\log \la)_{\bar{z}z}-\bar{h}_z+h(\log\la)_{\bar{z}}-|h|^2\right)U
\end{equation*}
yields
\begin{equation}\label{la1}\aligned
(\log \la)_{z\bar{z}}&=h_{\bar{z}}+\bar{h}_z-\la=-(\log\mu)_{z\bar{z}}-\la\\
&=\mu^2-2\la.
\endaligned
\end{equation}

Via Pl\"{u}cker coordinates, the Gauss map of $M$ can be written as
\begin{equation}
\g(p)=JX\w JT\w JE_1\w JE_2=-\sqrt{-1}(JX\w JT\w JW\w J\ol{W}).
\end{equation}
Thus
\begin{equation}\aligned
\g_*\pf{z}&=-\sqrt{-1}(JX\w JT\w JW_z\w J\ol{W}+JX\w JT\w JW\w J\ol{W}_z)\\
&=\sqrt{-1}\mu JX\w JT\w \ol{W}\w J\ol{W}
\endaligned
\end{equation}
and
\begin{equation}
\lan \g_* \pf{z},\g_* \pf{\bar{z}}\ran=\mu^2.
\end{equation}
This means
\begin{equation}
g:=2\mu^2|dz|^2
\end{equation}
is the induced metric on $\Si$, whose corresponding Gauss curvature is
\begin{equation}
K_g=-\f{(\log \mu^2)_{z\bar{z}}}{\mu^2}=\f{2(\mu^2-\la)}{\mu^2}=2-\f{2\la}{\mu^2}.
\end{equation}

Noting that $\pi_* \ol{W}=\bar{f}\pf{\bar{z}}$, we have
\begin{equation*}
\aligned
&\lan U,X\ran=\lan W_{\bar{z}},X\ran=\bar{f}^{-1}\lan \p_{\ol{W}}W,X\ran\\
=&-\bar{f}^{-1}\lan W,\p_{\ol{W}}X\ran=-\bar{f}^{-1}.
\endaligned
\end{equation*}
Combining with $\lan U,U\ran=0$, we can derive $U=-\bar{f}^{-1}(X+\pm\sqrt{-1}T)$, hence
\begin{equation}
\la=\lan U,\ol{U}\ran=2|f|^{-2}
\end{equation}
and
\begin{equation}
|B|^2=2|B(W,W)|^2=2|\lan B(W,W),JW\ran|^2=2|f\mu|^2=\f{4\mu^2}{\la}.
\end{equation}
In conjunction with (\ref{u4}) and (\ref{la1}), we arrive at
\begin{equation}\aligned
\De_g \log |B|^2&=\f{2}{\mu^2}(\log |B|^2)_{z\bar{z}}=\f{2}{\mu^2}\left(2(\log\mu)_{z\bar{z}}-(\log\la)_{z\bar{z}}\right)\\
&=\f{8\la}{\mu^2}-6=\f{32}{|B|^2}-6.
\endaligned
\end{equation}

In summary, we get a structure theorem as follows:
\begin{thm}\label{structure}
Let $M$ be a minimal Legendrian submanifold in $S^7$, such that $\la_3=0$ and $|B|^2\neq 0$ everywhere. Then the Gauss map
$\g$ is a submersion from $M$ onto a $2$-dimensional submanifold $\Si$ of $\grs{4}{4}$, the squared norm of the second fundamental form $|B|^2$
takes the same value in each fibre, and
\begin{equation}
\aligned
K_g&=2-\f{8}{|B|^2},\\
\De_g \log |B|^2&=\f{32}{|B|^2}-6.
\endaligned
\end{equation}
Here $g$ is the induced metric on $\Si$, $K_g$ and $\De_g$ are the Gauass curvature and the Laplacian operator on $(\Si,g)$, respectively.

On the other hand, if $(\Si,g)$ is a simply-connected $2$-dimensional Riemannian manifold, such that $K_g<2$ and
\begin{equation}
\De_g \log (2-K_g)=4K_g-2,
\end{equation}
then there exists an isometric immersion $\psi:\Si\ra \grs{4}{4}$ and a minimal Legendrian submanifold $M\subset S^7$,
such that:
\begin{itemize}
\item The rank of the Gauss map $\g:M\ra \grs{4}{4}$ is 2 everywhere (i.e. $\la_3=0$ and $|B|^2\neq 0$ everywhere on $M$);
\item The image manifold of the Gauss map is just $\Si$;
\item $|B|^2\equiv \f{8}{2-K_g}$ on each fibre of $\g$.
\end{itemize}

\end{thm}

In conjunction with the results in \cite{C-D-V-V,L-T}, we establish the following rigidity theorem:
\begin{thm}
Let $M$ be an $n$-dimensional compact minimal Legendrian submanifold in $S^{2n+1}$ with $n\geq 3$, such that
$\la_3\equiv 0$, then
\begin{itemize}
\item If $n\geq 4$, then $M$ has to be totally geodesic.
\item If $n=3$ and $|B|^2\neq 0$ everywhere, then $M$ is diffeomorphic to $S^3$.
\item If $n=3$ and $|B|^2\leq \f{16}{3}$, then $M$ is either totally geodesic or $|B|^2\equiv \f{16}{3}$ (i.e. $|B|^2+\la_2\equiv 8$).
For the latter case, $M=S^3=\{(z,w)\in \C^2:|z|^2+|w|^2=1\}$ and $F:M\ra \C^4$ is given as
\begin{equation*}
F(z,w)=(z^3+3z\bar{w}^2,\sqrt{3}(z^2 w+w\bar{w}^2-2z\bar{z}\bar{w}),
\sqrt{3}(zw^2+z\bar{z}^2-2w\bar{z}\bar{w}),w^3+3w\bar{z}^2).
\end{equation*}
\item If $n=3$ and $|B|^2\geq \f{16}{3}$, then $|B|^2\equiv \f{16}{3}$.
\end{itemize}
\end{thm}

	\bigskip\bigskip

	\bibliographystyle{amsplain}

\end{document}